\documentclass{amsart}
\usepackage{graphicx}
\usepackage{amssymb}
\usepackage{amsmath}
\usepackage{amsthm}
\usepackage{subfigure}
\usepackage{enumerate}
\usepackage{bbm}
\usepackage{xcolor}

\newtheorem{thm}{Theorem}[section]
\newtheorem{cor}[thm]{Corollary}

\newtheorem{lem}[thm]{Lemma}
\newtheorem{prop}[thm]{Proposition}
\theoremstyle{definition}
\newtheorem{defn}[thm]{Definition}
\newtheorem{rem}[thm]{Remark}

\newtheorem*{defn*}{Definition}
\newtheorem*{rems*}{Remarks}
\newtheorem*{rem*}{Remark}

\numberwithin{equation}{section}

\newcommand{\Eq}{{\text{E}}}

\newcommand{\M}{{M}}

\newcommand{\Css}{{\text{CSS}}}

\begin{document}

\title[Singular points of affine $\lambda$--equidistants] {Singular points of the Wigner caustic and affine equidistants of planar curves}
\author{Wojciech Domitrz, Micha\l{} Zwierzy\'nski}
\address{Faculty of Mathematics and Information Science\\
Warsaw University of Technology\\
ul. Koszykowa 75, 00-662 Warszawa\\
Poland
\\}

\email{domitrz@mini.pw.edu.pl, zwierzynskim@mini.pw.edu.pl}
\thanks{The work of W. Domitrz and M. Zwierzy\'nski was partially supported by NCN grant no. DEC-2013/11/B/ST1/03080. }



\begin{abstract}
In this paper we study singular points of the Wigner caustic and affine $\lambda$--equidistants of planar curves based on shapes of these curves. We generalize the Blaschke--S\"uss theorem on the existence of antipodal pairs of a~convex curve.
\end{abstract}

\maketitle

\section{Introduction}

The famous Blaschke-S\"uss theorem states that there are at least three pairs of antipodal pairs on an oval (\cite{G1, L1}). We recall that a pair of points on an oval is an antipodal pair if the tangent lines are parallel at these points and curvatures are equal. The Wigner caustic is a locus of midpoints of chords connecting points on a curve with parallel tangent lines (\cite{B1, DMR1, DR1, DZ-geometry, AH1}). The singular points of the Wigner caustic of an oval come exactly from the antipodal pairs. The Wigner caustic was first considered by M. Berry in \cite{B1} on the study of Wigner's phase-space representation of quantum states. Recently many interesting properties of the Wigner caustic were investigated. For instance the typical behavior of convex bodies in $\mathbb{R}^2$ were studied by R. Schneider using the middle hedgehog (\cite{S2, S3}) which is a natural generalization of the Wigner caustic for non--smooth ovals. Furthermore the absolute value of the oriented area of the Wigner caustic improves the classical isoperimetric inequality for ovals (\cite{Z2, Z3}) and improves the isoperimetric defect in the reverse isoperimetric inequality (\cite{CGR-Hurwitz-ineq}). Let us also notice the use of this set in the affine geometry. It leads to the construction of bi-dimensional improper affine spheres (\cite{CDR1}). The Wigner caustic is an example of an affine $\lambda$--equidistant, which is the locus of points dividing chords connecting points on $\M$ with parallel tangent lines in a fixed ratio $\lambda$ (\cite{DJRR1, DRR1, DZ-geometry, Z1}). 

In this paper we study singular points of Wigner caustic and affine equidistants of a planar curve. We prove theorems on existence of singular points based on the shape of the curve. In particular we find a generalization of the Blaschke-S\"uss theorem (see Theorem \ref{CorWCLoop}).

In Section \ref{SecGeneralProperties} we briefly sketch the known results on the Wigner caustic and affine equidistants. Then we proceed with the study of tangent lines and curvatures of affine equidistants and thanks to it we obtain many global results concerning properties of affine equidistants. Among other things we derive a formula for the number of their inflexion points, we prove that it is possible to reconstruct the original convex curve from any affine equidistants except the Wigner caustic and we show that any equidistant, except the Wigner caustic, (for a generic $\lambda$) of a generic closed convex curve has an even number of cusp singularities.

In Section \ref{SecExistence} theorems on existence of singularities of affine equidistants and the Wigner caustic are proved and some applications are indicated.

The methods used in this paper can be also applied to study singular points of the secant caustic (\cite{DRZ}).


\section{Properties of the Wigner caustic and affine equidistants}\label{SecGeneralProperties}

A smooth parametrized curve $M$ on the affine plane $\mathbb{R}^2$ is the image of $C^{\infty}$ smooth map $I\to\mathbb{R}^2$, where $I$ is an open interval in $\mathbb{R}$. The image of $C^{\infty}$ map $S^1\to\mathbb{R}^2$ is called a smooth \textit{closed} curve.  If the velocity of a smooth curve does not vanish, then the curve is called \textit{regular}. A curve with no intersection points is called \textit{simple}. A \textit{convex} curve is a regular simple closed curve if its curvature does not vanish. If there exists $\varepsilon>0$ such that $f\big((s_0-\varepsilon, s_0+\varepsilon)\big)$ is a $C^k$ smooth $1$ - dimensional manifold, where $k=1, 2, \ldots$ or $\infty$ and $(s_1,s_2)\ni s\mapsto f(s)\in\mathbb{R}^2$ is a parameterization of $M$, then a point $f(s_0)$ is called a \textit{$C^k$ regular point} of $M$. If a point $f(s_0)$ is not a $C^k$ regular for any $k>0$, then we call it a \textit{singular point}. A curve is called \textit{singular} if it has a singular point.


\begin{defn}\label{parallelpair}
A pair of different points $a, b$ in $\M$ is called a \textit{parallel pair} if the lines $T_aM$ and $T_bM$ are parallel.
\end{defn}

\begin{defn}\label{equidistantSet}
An \textit{affine} $\lambda$\textit{--equidistant} is the following set
$$\Eq_{\lambda}(\M)=\left\{\lambda a+(1-\lambda)b\ \big|\ a,b \text{ is a parallel pair of } \M\right\}.$$

The \textit{Wigner caustic} of $\M$ is $\Eq_{\frac{1}{2}}(\M)$.
\end{defn}

\begin{defn}
The \textit{Centre Symmetry Set} of $\M$ ($\Css(\M)$) is the envelope of lines joining parallel pairs of $\M$.
\end{defn}

Let us briefly sketch the known results on the geometry of affine $\lambda$--equidistants and the Centre Symmetry Set (see Figure \ref{FigCssEq}. Let $\M$ be a generic convex curve. Then $\Css(\M)$, $E_{\frac{1}{2}}(M)$ and $\Eq_{\lambda}(\M)$ for a generic $\lambda$ are smooth closed curves which can have only cusp singularities (\cite{B1, GH1, GZ1, J1}). Furthermore regular parts of $\Css(\M)$ are formed by cusp singularities of $\Eq_{\lambda}(\M)$ (\cite{GZ1}). Both: the number of cusp singularities of $E_{\frac{1}{2}}(M)$ and the number of cusp singularities of $\Css(M)$, are odd and not smaller than $3$ (\cite{B1, GH1, G1}) and the number of cusp singularities of $E_{\frac{1}{2}}(M)$ is not greater than the number of cusp singularities of $\Css(M)$. We will show that the number of cusps of $\Eq_{\lambda}(\M)$ for a generic $\displaystyle\lambda\neq\frac{1}{2}$ is even (Theorem \ref{PropCuspsEq}). Bifurcations of affine $\lambda$--equidistants were studied in details in \cite{GR1, GWZ1, JJR1}. The global geometry of the Wigner caustic and other affine $\lambda$--equidistants were studied in \cite{DZ-geometry}. The geometry of an affine extended wave front, i.e. the set $\bigcup_{\lambda\in[0,1]}\{\lambda\}\times E_{\lambda}(M)$ was studied in \cite{DZ-GBtheorem}. Let $p$ be a inflexion point of $\M$. Then $\Css(\M)$ is tangent to this inflexion point and has an endpoint there. The set $\Eq_{\lambda}(\M)$ for $\displaystyle\lambda\neq \frac{1}{2}$ has an inflexion point at $p$ (as the limit point) and is tangent to $\M$ at $p$. The Wigner caustic is tangent to $\M$ at $p$ too and it has an endpoint there. The Wigner caustic and the Centre Symmetry Set approach $p$ from opposite sides (\cite{B1, DMR1, GH1, GWZ1}). The geometry of this branch of the Wigner caustic, which is called the Wigner caustic on shell, is studied in \cite{DZ-geometry}.

\begin{figure}[h]
\centering
\includegraphics[scale=0.23]{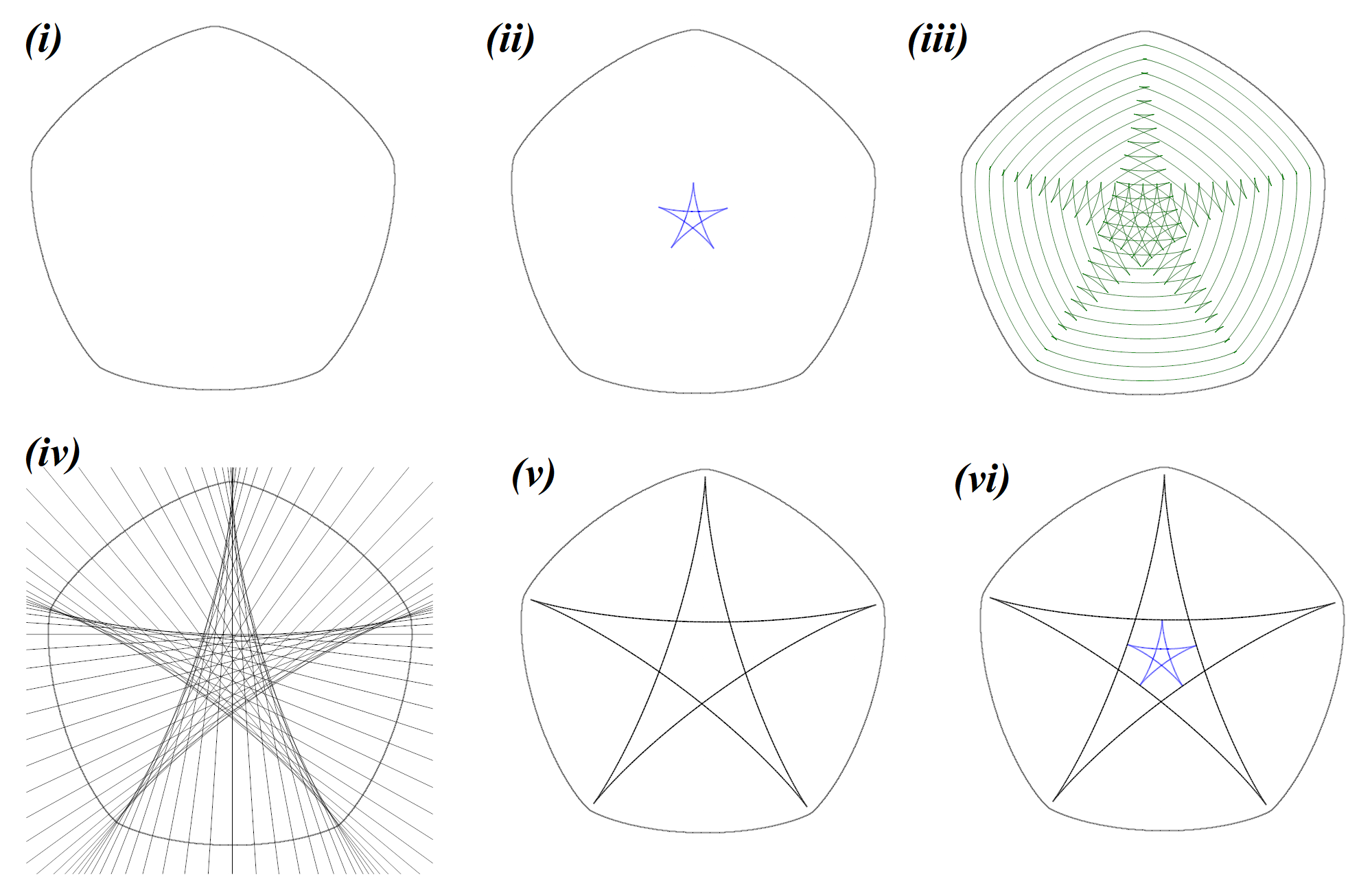}
\caption{(i) A curve $M$, (ii) $E_{\frac{1}{2}}(M)$, (iii) $E_{\lambda}(M)$ for $\displaystyle\lambda=\frac{k}{26}$ for $k=1,2,\ldots,13$, (iv) lines joining parallel pairs of $M$, (v) $\Css(M)$, (vi) $E_{\frac{1}{2}}(M)$ and $\Css(M)$.}
\label{FigCssEq}
\end{figure}

Let us denote by $\kappa_{\M}(p)$ the signed curvature of $\M$ at $p$.

If $a,b$ is a parallel pair of $\M$ such that local parameterizations of $\M$ nearby points $a$ and $b$ are in opposite directions and $\kappa_{\M}(a)+\kappa_{\M}(b)\neq 0$, then the point $\displaystyle \frac{\kappa_{\M}(a)\cdot a+\kappa_{\M}(b)\cdot b}{\kappa_{\M}(a)+\kappa_{\M}(b)}$ belongs to $\Css(\M)$ (\cite{GH1}).

By direct calculations we get the following lemma.

\begin{lem}\label{LemParallelCurvature}
Let $\M$ be a closed regular curve. Let $a,b$ be a parallel pair of $\M$, such that $\M$ is parameterized at $a$ and $b$ in opposite directions and $\kappa_{\M}(b)\neq 0$. Let $p=\lambda a+(1-\lambda)b$ be a regular point of $\Eq_{\lambda}(\M)$ and let $\displaystyle q=\frac{\kappa_{\M}(a)\cdot a+\kappa_{\M}(b)\cdot b}{\kappa_{\M}(a)+\kappa_{\M}(b)}$ be a regular point of $\Css(\M)$.

Then 
\begin{enumerate}[(i)]
\item the tangent line to $\Eq_{\lambda}(\M)$ at $p$ is parallel to the tangent lines to $\M$ at $a$ and $b$.

\item the curvature of $\Eq_{\lambda}(\M)$ at $p$ is equal to

$$\kappa_{\Eq_{\lambda}(\M)}(p)=\frac{\kappa_{\M}(a)|\kappa_{\M}(b)|}{\left|\lambda\kappa_{\M}(b)-(1-\lambda)\kappa_{\M}(a)\right|}.$$

\item the curvature of $\Css(\M)$ at $q$ is equal to

$$\kappa_{\Css(\M)}(q)=\mbox{sgn}(\kappa_{\M}(b))\cdot \frac{\displaystyle (\kappa_{\M}(a)+\kappa_{\M}(b))^3}{\displaystyle \left|\kappa_{\M}^2(b)\kappa'_{\M}(a)-\kappa_{\M}^2(a)\kappa'_{\M}(b)\right|}\cdot\frac{\displaystyle \det\left(a-b,\mathbbm{t}(a)\right)}{|a-b|^3},$$
where $\mathbbm{t}(a)$ is a tangent unit vector field compatible with the orientation of $\M$ at $a$.
\end{enumerate}
\end{lem}

Lemma \ref{LemParallelCurvature}(ii)--(iii) implies the following propositions.

\begin{prop}\label{PropSingularPointOfEq} \cite{GZ1}
Let $a,b$ be a parallel pair of a generic regular curve $\M$, such that $\M$ is parameterized at $a$ and $b$ in opposite directions and $\kappa_{\M}(b)\neq 0$. Then the point $\lambda a+(1-\lambda)b$ is a singular point of $\Eq_{\lambda}(\M)$ if and only if $\displaystyle\frac{\kappa_{\M}(a)}{\kappa_{\M}(b)}=\frac{\lambda}{1-\lambda}$.
\end{prop}

\begin{prop}\cite{DZ-geometry}\label{PropInflOfEq}
Let $a, b$ be a parallel pair of a generic regular closed curve $\M$ and let $\lambda\neq 0, 1$. Then $\lambda a+(1-\lambda)b$ and $(1-\lambda)a+\lambda b$ are inflexion points of $\Eq_{\lambda}(\M)$ if and only if one of the points $a$, $b$ is an inflexion point of $\M$.
\end{prop}

\begin{cor}\label{CorConvexNoInf}
If $\M$ is a convex curve, then $\Eq_{\lambda}(\M)$ for $\lambda\in\mathbb{R}$ and $\Css(\M)$ have no inflexion points.
\end{cor}

Let $\tau_{p}$ denote the translation by a vector $p\in\mathbb{R}^2$.

\begin{defn}
A curve $\M$ is \textit{curved in the same side at $a$ and $b$} (resp. \textit{curved in the different sides}), where $a, b$ is a parallel pair of $\M$, if the center of curvature of $\M$ at $a$ and the center of curvature of  $\tau_{a-b}(\M)$ at $a=\tau_{a-b}(b)$ lie on the same side (resp. on the different sides) of the tangent line to $\M$ at $a$.
\end{defn}

In Figure \ref{FigCurved}(i) we illustrate a curve curved in the same side at a parallel pair $a,b$ and in Figure \ref{FigCurved}(ii) we illustrate a curve curved in the different sides at a parallel pair $a,b$.

\begin{figure}[h]
\centering
\includegraphics[scale=0.5]{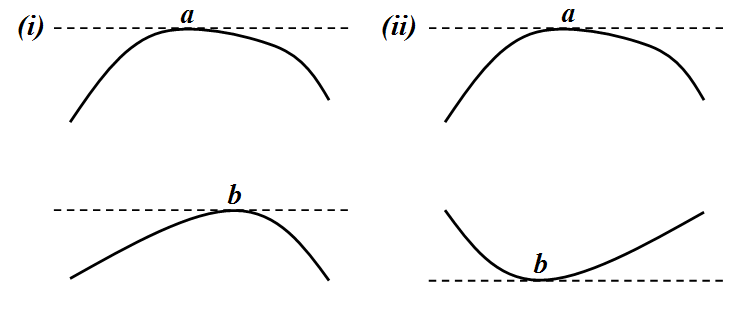}
\caption{}
\label{FigCurved}
\end{figure}

\begin{prop}\cite{DZ-geometry}\label{PropRegularPointOfEq}
\begin{enumerate}[(i)]
\item If $\M$ is curved in the same side at a parallel pair $a,b$, then $\lambda a+(1-\lambda)b$ is a $C^{\infty}$ regular point of $\Eq_{\lambda}(\M)$ for $\lambda\in(0,1)$.
\item If $\M$ is curved in the different sides at a parallel pair $a,b$, then $\lambda a+(1-\lambda)b$ is a $C^{\infty}$ regular point of $\Eq_{\lambda}(\M)$ for $\lambda\in(-\infty,0)\cup(1,\infty)$.
\end{enumerate}
\end{prop}
 
\begin{defn}
The \textit{tangent line of} $\Eq_{\lambda}(\M)$ \textit{at a cusp point} $p$ is the limit of a sequence of $1$--dimensional vector spaces $T_{q_n}\M$ in $\mathbb{R}P^1$ for any sequence $q_n$ of regular points of $\Eq_{\lambda}(\M)$ converging to $p$.
\end{defn}

By Lemma \ref{LemParallelCurvature}(i) the tangent line to $\Eq_{\lambda}(\M)$ at the cusp point $\lambda a+(1-\lambda)b$ is parallel to the tangent lines to $\M$ at $a$ and $b$.

\begin{defn}\label{DefNormalVectorEq}
Let $\displaystyle\lambda\neq\frac{1}{2}$. The normal vector field to $\Eq_{\lambda}(\M)$ at the point \linebreak $\lambda a+(1-\lambda)b$ is equal to the normal vector field to $\M$ at the point $a$.
\end{defn}

\begin{rem}
If $\M$ is convex we have the well defined continuous normal vector field on the double covering $\M$ of $\Eq_{\frac{1}{2}}(\M)$, namely: $\displaystyle\M\ni a\mapsto\frac{a+b}{2}\in\Eq_{\frac{1}{2}}(\M)$.
\end{rem}

Let us notice that the continuous normal vector field to $\Eq_{\lambda}(\M)$ at regular and cusp points is perpendicular to the tangent line to $\Eq_{\lambda}(\M)$. Using this fact and the above definition we define the rotation number in the following way.

\begin{defn}\label{DefRotationNumber}
The \textit{rotation number} of a smooth curve $\M$ with well defined continuous normal vector field is the rotation number of this vector field. 
\end{defn}

Moreover by Lemma \ref{LemParallelCurvature}(i) we can easily get the next two propositions.

\begin{prop}
Let $\M$ be a generic regular closed curve in $\mathbb{R}^2$. If $l$ is a bitangent line to $\M$ at points $a,b$, then $\Eq_{\lambda}(\M)$ is tangent to $l$ at points $\lambda a+(1-\lambda)b$ and $(1-\lambda)a+\lambda b$.
\end{prop}

If $\M$ is a convex curve, then for any line $l$ in $\mathbb{R}^2$ there exist exactly two points $a,b\in\M$ in which the tangent lines to $\M$ at $a$ and $b$ are parallel to $l$. Also there exists exactly one parallel pair $a,b$ of $\M$ such that chord passing through $a$ and $b$ is parallel to $l$. Therefore we get the following proposition. 

\begin{prop}\label{PropNumberOfTangentLines}
Let $\M$ be a generic convex curve. Then for any line $l$ in $\mathbb{R}^2$  
\begin{enumerate}[(i)]
\item there exists exactly one point $p$ in $ \Eq_{\frac{1}{2}}(\M)$ such that the tangent line to $\Eq_{\frac{1}{2}}(\M)$ at $p$ is parallel to $l$.
\item there exist exactly two different points $p_1,p_2$ in $\Eq_{\lambda}(\M)$ for $\lambda\neq \frac{1}{2}$, such that the tangent lines to $\Eq_{\lambda}(\M)$ at $p_1$ and $p_2$ are parallel to $l$.
\item there exists exactly one point $p$ in $\Css(\M)$ such that the tangent line to $\Css(\M)$ at $p$ is parallel to $l$ (\cite{GH1}).
\end{enumerate}
\end{prop}

\begin{thm}\label{PropCuspsEq}
Let $\M$ be a generic convex curve. Then the number of cusps of $\Eq_{\lambda}(\M)$ for a generic $\displaystyle\lambda\neq \frac{1}{2}$ is even and the number of cusps of $\Eq_{\frac{1}{2}}(\M)$ is odd.
\end{thm}

\begin{proof}

For a generic convex curve $\M$, $\Eq_{\frac{1}{2}}(\M)$ has got only cusp singularities and the same statement holds for $\Eq_{\lambda}(\M)$ for a generic $\displaystyle\lambda\neq\frac{1}{2}$. Without loss of generality we may assume that $\M$ is positively oriented.

Let $P: \M\to \M$ maps a point $a$ to the point $b\neq a$ such that $a,b$ is a parallel pair. This map is well defined on a convex curve $\M$. If $f: S^1\to\mathbb{R}^2$ is a parameterization of $\M$, then $S^1\ni s\mapsto \lambda f(s)+(1-\lambda)P(f(s))\in\Eq_{\lambda}(\M)$ is a parameterization of $\Eq_{\lambda}(\M)$ for $\displaystyle\lambda\neq\frac{1}{2}$ and it is a double covering of $\Eq_{\frac{1}{2}}(\M)$. The normal vector to $\Eq_{\lambda}(\M)$ at $\lambda f(s)+(1-\lambda)P(f(s))$ is the normal vector to $\M$ at $f(s)$, so the rotation number of $\Eq_{\lambda}(\M)$ for $\displaystyle\lambda\neq\frac{1}{2}$ is equal to the rotation number of $\M$, which is $1$ for positively oriented convex curves, and the rotation number of $\Eq_{\frac{1}{2}}(\M)$ is equal to $\displaystyle\frac{1}{2}$. 

A continuous normal vector field to the germ of a curve with the cusp singularity is directed outside the cusp on the one of two connected regular components and is directed inside the cusp on the other component as it is illustrated in Figure \ref{PictureNormalVectorToCusp}. Since the rotation number of $\Eq_{\lambda}(\M)$ for $\displaystyle\lambda\neq\frac{1}{2}$ is $1$, the number of cusps of $\Eq_{\lambda}(\M)$ for $\displaystyle\lambda\neq\frac{1}{2}$ is even and the number of cusps of $\Eq_{\frac{1}{2}}(\M)$ is odd because the rotation number of $\Eq_{\frac{1}{2}}(\M)$ is equal to $\displaystyle\frac{1}{2}$. The last statement has been proved by M. Berry in \cite{B1} using a different method.

\begin{figure}[h]
\centering
\includegraphics[scale=0.16]{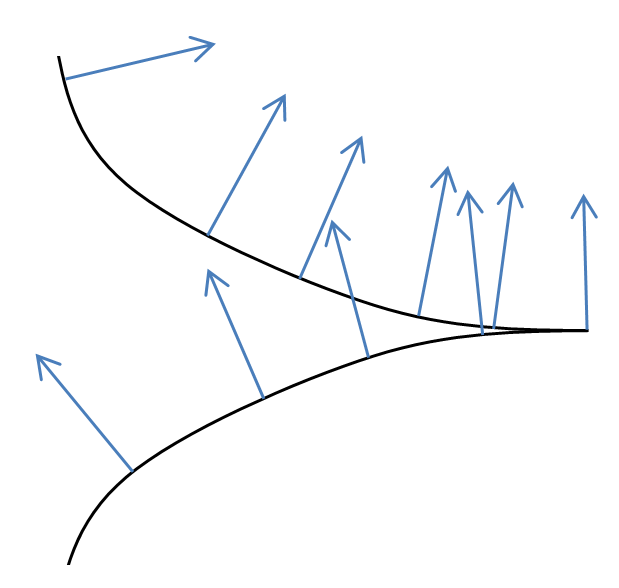}
\caption{An ordinary cusp singularity with a continuous normal vector field. Vectors in upper regular component of a curve are directed outside the cusp, others are directed inside the cusp.}
\label{PictureNormalVectorToCusp}
\end{figure}

\end{proof}

In \cite{B1, L1} it is proved by analytical methods that the number of cusp singularities of the Wigner caustic of a generic convex curve is odd and not smaller than $3$. We present an elementary proof of this fact for the Wigner caustic and the Centre Symmetry Set.

\begin{prop}\label{PropCssCant1Cusp}
Let $\M$ be a generic convex curve. Then the number of cusps of $\Css(\M)$ and $\Eq_{\frac{1}{2}}(\M)$ is not smaller than $3$.
\end{prop}
\begin{proof}

\begin{figure}[h]
\centering
\includegraphics[scale=0.18]{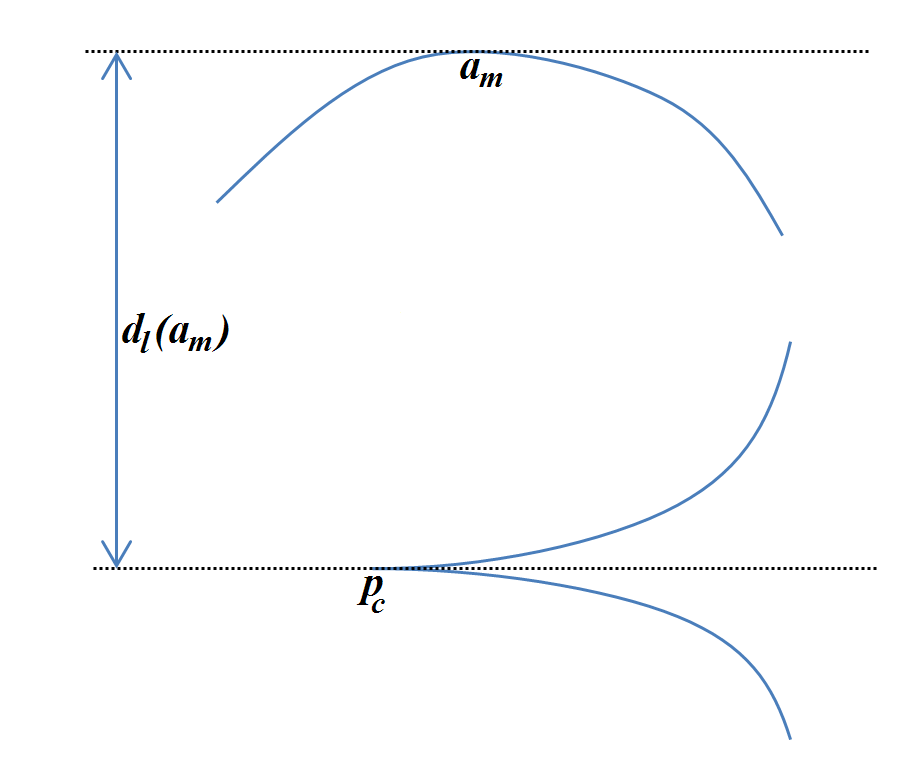}
\caption{The cusp ($p_c$), dotted lines are parallel.}
\label{PictureCusp}
\end{figure}

The number of cusps of the Wigner caustic is odd (see Theorem \ref{PropCuspsEq} or \cite{B1}). Let us assume that the Wigner caustic has got exactly one cusp. Let $p_c$ be the cusp point. Let us consider the line $l$, which is tangent to $\Eq_{\frac{1}{2}}(\M)$ at $p_c$. Let $d_l(a)$ be the distance between a point $a\in\Eq_{\frac{1}{2}}(\M)$ and $l$. Since $d_l(p_c)=0$ and every point $a\in\Eq_{\frac{1}{2}}(\M)$ such that $a\neq p_c$ is a regular point of $\Eq_{\frac{1}{2}}(\M)$, there exists a point $a_m$ such that $d_l(a_m)$ is a maximum value of $d_l$ (see Figure \ref{PictureCusp}). Then the tangent line to the Wigner caustic at $a_m$ is parallel to $l$, which is impossible by Proposition \ref{PropNumberOfTangentLines}(i).

By Proposition \ref{PropNumberOfTangentLines}(iii), the proof for $\Css(\M)$ is similar.

\end{proof}

\pagebreak

\begin{prop}\label{ThmEqForEq}
Let $\M$ be a generic convex curve. 
\begin{enumerate}[(i)]
\item If $\displaystyle\lambda\neq\frac{1}{2}$, then
$$\Eq_{\delta}\left(\Eq_{\lambda}(\M)\right)=\Eq_{\delta(1-\lambda)+\lambda(1-\delta)}(\M).$$
\item If $\lambda\neq\frac{1}{2}, \delta\neq\frac{1}{2}$, then
$$\Css(\Eq_{\lambda}(\M))=\Css(\Eq_{\delta}(\M)).$$
\end{enumerate}
\end{prop}
\begin{proof}Let $a,b$ be a parallel pair of $\M$. By Theorem \ref{LemParallelCurvature}(i) $p=\lambda a+(1-\lambda)b$, $q=(1-\lambda)a+\lambda b$ is a parallel pair of $\Eq_{\lambda}(\M)$ and by Proposition \ref{PropNumberOfTangentLines} there are no more points on $\Eq_{\lambda}(\M)$ for which tangent lines at them are parallel to tangent lines at $p, q$. Then $r=\delta p+(1-\delta)q$ is a point of $\Eq_{\delta}\left(\Eq_{\lambda}(\M)\right)$ and one can get that $r=(1-\Lambda)a+\Lambda b$ where $\Lambda=\delta(1-\lambda)+\lambda(1-\delta)$.
Any equidistant of $\M$, except the Wigner caustic, has the same collection of chords joining parallel pairs on them, then the Centre Symmetry Set  of any equidistant, as an envelope of these lines, is the same set.

\end{proof}

By Proposition \ref{ThmEqForEq} we get the following corollary on reconstruction of the original curve from its affine equidistant and the stability of the Wigner caustic.
\begin{cor}\label{CorReconstruction}
 Let $\M$ be a generic convex curve. If $\lambda\neq\frac{1}{2}$, then
\begin{align*}
& M	=\Eq_{-\lambda(1-2\lambda)^{-1}}(\Eq_{\lambda}(\M))=\Eq_{(1-\lambda)(1-2\lambda)^{-1}}(\Eq_{\lambda}(\M)),\\
& \Eq_{\frac{1}{2}}(\Eq_{\lambda}(\M)) =\Eq_{\frac{1}{2}}(\M).
\end{align*}
\end{cor}

\section{Existence of a singular point of an affine $\lambda$--equidistant.}\label{SecExistence}

In this section we prove theorems on existence of singularities of affine equidistants and the Wigner caustic. We also present some applications of them.

\begin{figure}[h]
\centering
\includegraphics[scale=0.28]{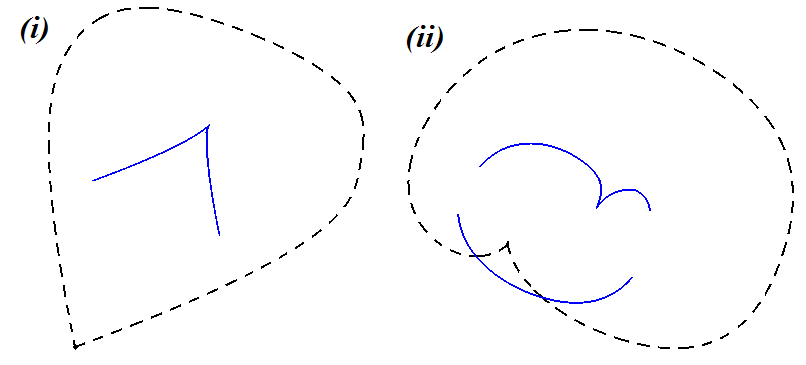}
\caption{(i) A convex loop $L$ (the dashed line) and $\Eq_{\frac{1}{2}}(L)$, (ii) a non--convex loop $L$ (the dashed line) and $\Eq_{\frac{1}{2}}(L)$.}
\label{PictureLoops}
\end{figure}

\begin{defn}
A simple smooth curve $\gamma: (s_1,s_2)\to\mathbb{R}^2$ with non-vanishing curvature is called a \textit{loop} if $\displaystyle\lim_{s\to s_1^+}\gamma(s)=\lim_{s\to s_2^-}\gamma(s)$. A loop $\gamma$ is called \textit{convex} if the absolute value of its rotation number is not greater than $1$, otherwise it is called \textit{non--convex}.
\end{defn}

We illustrate examples of loops in Figure \ref{PictureLoops}.

\begin{thm}\label{CorWCLoop}
The Wigner caustic of a loop has a singular point.
\end{thm}

Theorem \ref{CorWCLoop} follows from the more general, but technical, propositions: Proposition \ref{ThmSingPointsGeneralization1} in the case of convex loops and Proposition \ref{PropSingPointsGeneralization3} in the case of non-convex loops.

\begin{figure}[h]
\centering
\includegraphics[scale=0.25]{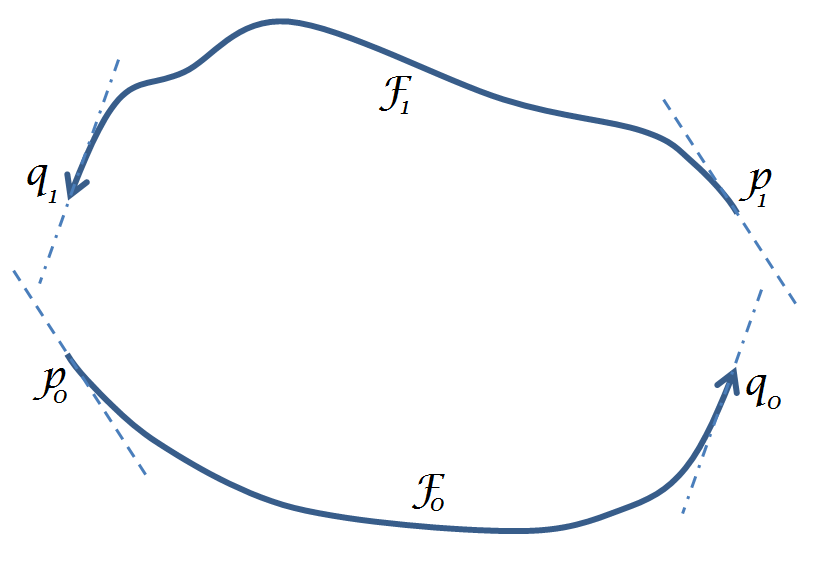}
\caption{}
\label{FigWrinkels1}
\end{figure}

\begin{prop}\label{ThmSingCurv1}
Let $\mathcal{F}_0$, $\mathcal{F}_1$ be embedded curves with endpoints $p_0, q_0$ and $p_1, q_1$, respectively, such that
\begin{enumerate}[(i)]
\item $p_0, p_1$ and $q_0, q_1$ are parallel pairs,
\item for every point $a_1$ in $\mathcal{F}_1$ there exists a point $a_0$ in $\mathcal{F}_0$ such that $a_0, a_1$ is a parallel pair,
\item $\kappa_{\mathcal{F}_1}(p_1)$, $\kappa_{\mathcal{F}_1}(q_1)$ and $\kappa_{\mathcal{F}_0}(a_0)$ for every $a_0\in\mathcal{F}_0$ are positive,
\item the rotation number of $\mathcal{F}_0$ is smaller than $\displaystyle\frac{1}{2}$,
\item $\mathcal{F}_0, \mathcal{F}_1$ are curved in the different sides at $p_0, p_1$ and $q_0, q_1$.
\end{enumerate}
If $\rho_{\min}$ (resp. $\rho_{\max}$) is the minimum (resp. the maximum) of the set \linebreak $\displaystyle\left\{\frac{\kappa_{\mathcal{F}_1}(p_1)}{\kappa_{\mathcal{F}_0}(p_0)}, \frac{\kappa_{\mathcal{F}_1}(q_1)}{\kappa_{\mathcal{F}_0}(q_0)}\right\}$ then the set $\Eq_{\lambda}(\mathcal{F}_0\cup\mathcal{F}_1)$ has a singular point for every \linebreak $\displaystyle\lambda\in\left[\frac{\rho_{\min}}{1+\rho_{\min}}, \frac{\rho_{\max}}{1+\rho_{\max}}\right]\cup\left[1-\frac{\rho_{\max}}{1+\rho_{\max}}, 1-\frac{\rho_{\min}}{1+\rho_{\min}}\right]$. 
\end{prop} 
\begin{proof}
Since $\Eq_{\lambda}(\mathcal{F}_0\cup\mathcal{F}_1)=\Eq_{1-\lambda}(\mathcal{F}_0\cup\mathcal{F}_1)$, let us fix $\lambda$ in $\displaystyle\left[\frac{\rho_{\min}}{1+\rho_{\min}}, \frac{\rho_{\max}}{1+\rho_{\max}}\right]$. Let $g:[t_0,t_1]\to\mathbb{R}^2, f:[s_0,s_1]\to\mathbb{R}^2$ be the arc length parameterizations of $\mathcal{F}_0, \mathcal{F}_1$, respectively. From (ii)-(iii) we get that for every point $a_1$ in $\mathcal{F}_1$ there exists exactly one point $a_0$ in $\mathcal{F}_0$ such that $a_0, a_1$ is a parallel pair. Therefore there exists a function $t:[s_0,s_1]\to[t_0,t_1]$ such that
\begin{align}\label{ParallelPropertyThm31}
f'(s)=-g'(t(s)).
\end{align}
By the implicit function theorem the function $t$ is smooth and $\displaystyle t'(s)=\frac{\kappa_{\mathcal{F}_1}(f(s))}{\kappa_{\mathcal{F}_0}(g(t(s)))}$. Let $\displaystyle\varrho=\frac{\lambda}{1-\lambda}$. Then $\varrho$ belongs to $[\rho_{\min}, \rho_{\max}]$. By Darboux Theorem there exists $s\in[s_0,s_1]$ such that $t'(s)=\varrho$. Then by Proposition \ref{PropSingularPointOfEq} the set $\Eq_{\lambda}(\mathcal{F}_0\cup\mathcal{F}_1)$ has a singular point for $\displaystyle\lambda=\frac{\varrho}{1+\varrho}$.

\end{proof}

Arcs satisfying assumptions of Proposition \ref{ThmSingCurv1} are illustrated in Figure \ref{FigWrinkels1}.

By the same argument we get the following result.

\begin{prop}\label{PropSingCurv2}
Under the assumptions of Proposition \ref{ThmSingCurv1} if we replace (iv) with
\begin{itemize}
\item[\textit{(iv)}] $\mathcal{F}_0, \mathcal{F}_1$ are curved in the same side at $p_0, p_1$ and $q_0, q_1$,
\end{itemize}
then the set $\Eq_{\lambda}(\mathcal{F}_0\cup\mathcal{F}_1)$ has a singular point
\begin{itemize}
\item for every $\displaystyle
\lambda\in\left[\frac{\rho_{\min}}{\rho_{\min}-1},\frac{\rho_{\max}}{\rho_{\max}-1} \right]\cup\left[1-\frac{\rho_{\max}}{\rho_{\max}-1}, 1-\frac{\rho_{\min}}{\rho_{\min}-1}\right]$ \linebreak if $\rho_{\min}>1$,
\item for every $\displaystyle
\lambda\in\left(-\infty, 1-\frac{\rho_{\max}}{\rho_{\max}-1}\right]\cup\left[\frac{\rho_{\max}}{\rho_{\max}-1}, \infty\right)$
if $\rho_{\min}<1<\rho_{\max}$.
\end{itemize} 
\end{prop}

Proposition \ref{CorSingCurv1} and Proposition \ref{CorSingCurv2} are the limiting versions of Proposition \ref{ThmSingCurv1} and Proposition \ref{PropSingCurv2}, respectively.

\begin{prop}\label{CorSingCurv1}
Let $\mathcal{F}_0$, $\mathcal{F}_1$ be embedded curves with endpoints $p_0, q_0$ and $p_1, q_1$, respectively, such that
\begin{enumerate}[(i)]
\item $p_0, p_1$ and $q_0, q_1$ are parallel pairs,
\item for every point $a_1$ in $\mathcal{F}_1$ there exists a point $a_0$ in $\mathcal{F}_0$ such that $a_0, a_1$ is a parallel pair,
\item $\kappa_{\mathcal{F}_0}(a)>0$ for every $a\neq p_0$, $\kappa_{\mathcal{F}_0}(p_0)=0$, $\kappa_{\mathcal{F}_1}(q_1)=0$ and $\kappa_{\mathcal{F}_1}(p_1)>0$,
\item the rotation number of $\mathcal{F}_0$ is smaller than $\displaystyle\frac{1}{2}$,
\item $\mathcal{F}_0, \mathcal{F}_1$ are curved in the different sides at parallel pairs $a_0, a_1$ close to $p_0, p_1$ and $q_0, q_1$, respectively. 
\end{enumerate}
Then the set $\Eq_{\lambda}(\mathcal{F}_0\cup\mathcal{F}_1)$ has a singular point for every $\lambda\in(0,1)$.
\end{prop}

\begin{prop}\label{CorSingCurv2}
Under the assumptions of Proposition \ref{CorSingCurv1} if we replace (iv) with
\begin{itemize}
\item[\textit{(iv)}] $\mathcal{F}_0, \mathcal{F}_1$ are curved in the same side at parallel pairs $a_0, a_1$ close to $p_0, p_1$ and $q_0, q_1$, respectively. 
\end{itemize}
Then the set $\Eq_{\lambda}(\mathcal{F}_0\cup\mathcal{F}_1)$ has a singular point for every $\lambda\in(-\infty,0)\cup(1,\infty)$.
\end{prop}

The curvature of a curve at a point is not an affine invariant but the ratio of curvatures at parallel points is an affine invariant. Inflexion points are affine invariants and they are easy to detect. Thus Proposition \ref{CorSingCurv1} and Proposition \ref{CorSingCurv2} are useful to study singularities of affine $\lambda$--equidistants (see Figure \ref{FigTwoInflPts}). In Figure \ref{FigTwoInflPts}(i) and in Figure \ref{FigTwoInflPts}(ii) we present arcs which satisfy the assumptions of Propositions \ref{CorSingCurv1} and Proposition \ref{CorSingCurv2}, respectively. Notice that $p_0$ and $q_1$ are inflexion points. By Proposition \ref{CorSingCurv1} for all $\lambda\in (0,1)$ an affine $\lambda$--equidistant of a set $\mathcal{F}_0\cup\mathcal{F}_1$  in Figure \ref{FigTwoInflPts}(i)  is singular. Similarly, for all $\lambda\in(-\infty,0)\cup(1,\infty)$ an affine $\lambda$--equidistant of a set $\mathcal{F}_0\cup\mathcal{F}_1$ in Figure \ref{FigTwoInflPts}(ii) is singular. Furthermore, in Figure \ref{FigTwoInflPts}(iii) we present arcs $\mathcal{F}_1$, $\mathcal{F}_2$, $\mathcal{F}_3$, $\mathcal{F}_4$ of one curve such that:
\begin{itemize}
\item arcs $\mathcal{F}_1$, $\mathcal{F}_4$ and arcs $\mathcal{F}_2$, $\mathcal{F}_3$ satisfy the assumptions of Proposition \ref{CorSingCurv1}, 
\item arcs $\mathcal{F}_1$, $\mathcal{F}_3$ and arcs $\mathcal{F}_2$, $\mathcal{F}_4$ satisfy the assumptions of Proposition \ref{CorSingCurv2}, 
\end{itemize}
therefore for all $\lambda\neq 0, 1$ an affine $\lambda$--equidistant of a set $\mathcal{F}_1\cup\mathcal{F}_2\cup\mathcal{F}_3\cup\mathcal{F}_4$ is singular. In Figure \ref{FigTwoInflPts}(iv) we illustrate a curve from Figure \ref{FigTwoInflPts}(iii) together with an affine  $\displaystyle\frac{1}{5}$--equidistant of $\mathcal{F}_1\cup\mathcal{F}_2\cup\mathcal{F}_3\cup\mathcal{F}_4$.

\begin{figure}[h]
\centering
\includegraphics[scale=0.29]{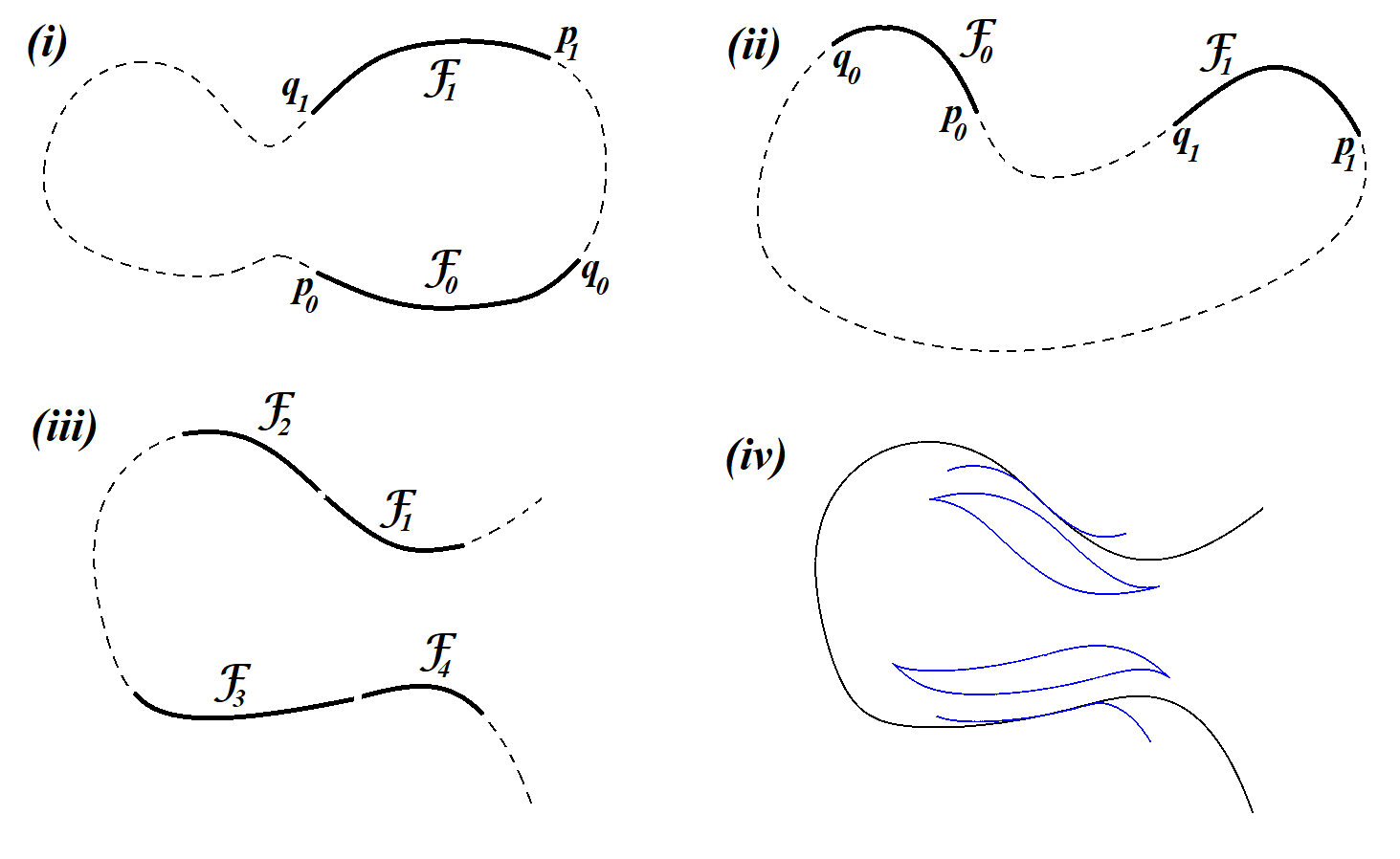}
\caption{}
\label{FigTwoInflPts}
\end{figure}

\begin{prop}\label{ThmSingPointsGeneralization1}
Let $\mathcal{F}_0$ and $\mathcal{F}_1$ be embedded regular curves with endpoints $p, q_0$ and $p, q_1$, respectively. Let $l_0$ be the line through $q_1$ parallel to $T_{p}\mathcal{F}_0$ and let $l_1$ be the line through $q_0$ parallel to $T_{p}\mathcal{F}_1$. Let $c=l_0\cap l_1$, $b_0=l_0\cap T_p\mathcal{F}_1$, $b_1=l_1\cap T_p\mathcal{F}_0$. Let us assume that
\begin{enumerate}[(i)]
\item $T_{p}\mathcal{F}_0 \| T_{q_1}\mathcal{F}_1$ and $T_{q_0}\mathcal{F}_0 \| T_{p}\mathcal{F}_1$,
\item the curvature of $\mathcal{F}_i$ for $i=0,1$ does not vanish at any point,
\item absolute values of rotation numbers of $\mathcal{F}_0$ and $\mathcal{F}_1$ are the same and smaller than $\displaystyle\frac{1}{2}$,
\item for every point $a_i$ in $\mathcal{F}_i$, there is exactly one point $a_{j}$ in $\mathcal{F}_{j}$ such that $a_i,a_{j}$ is a parallel pair for $i\neq j$,
\item $\mathcal{F}_0$, $\mathcal{F}_1$ are curved in the different sides at every parallel pair $a_0, a_1$ such that $a_i\in\mathcal{F}_i$ for $i=0,1$.
\end{enumerate}

Let $\rho_{\max}$ (resp. $\rho_{\min}$) be the maximum (resp. the minimum) of the set \linebreak $\displaystyle\left\{\frac{c-b_1}{q_1-b_1}, \frac{c-b_0}{q_0-b_0}\right\}.$

If $\rho_{\max}<1$ then the set $\Eq_{\lambda}(\mathcal{F}_0\cup\mathcal{F}_1)$ has a singular point for every \linebreak $\displaystyle\lambda\in\left[\frac{\rho_{\max}}{\rho_{\max}+1}, \frac{1}{\rho_{\max}+1}\right]$.

If $\rho_{\min}>1$ then the set $\Eq_{\lambda}(\mathcal{F}_0\cup\mathcal{F}_1)$ has a singular point for every \linebreak $\displaystyle\lambda\in\left[\frac{1}{\rho_{\min}+1}, \frac{\rho_{\min}}{\rho_{\min}+1}\right]$.

In particular, if $\rho_{\max}<1$ or $\rho_{\min}>1$ then the Wigner caustic of $\mathcal{F}_0\cup\mathcal{F}_1$ has a singular point.
\end{prop}
\begin{proof}
Let us consider the case $\rho_{\max}<1$, the proof of the case $\rho_{\min}>1$ is similar.

Since $\Eq_{\lambda}(\mathcal{F}_0\cup\mathcal{F}_1)=\Eq_{1-\lambda}(\mathcal{F}_0\cup\mathcal{F}_1)$, to finish the proof it is enough to show that $\Eq_{\lambda}(\mathcal{F}_0\cup\mathcal{F}_1)$ has a singular point for $\displaystyle\lambda\in\left[\frac{\rho_{\max}}{1+\rho_{\max}}, \frac{1}{2}\right]$.

Let us fix $\lambda$ in the interval $\displaystyle\left[\frac{\rho_{\max}}{1+\rho_{\max}}, \frac{1}{2}\right]$.

We can transform the parallelogram bounded by $T_{p}\mathcal{F}_0, T_{q_0}\mathcal{F}_0, T_{p}\mathcal{F}_1, T_{q_1}\mathcal{F}_1$ to the unit square by an affine transformation $A:\mathbb{R}^2\to\mathbb{R}^2$ (see Figure \ref{PictureThmParallelogramGeneralization}). Since any affine $\lambda$--equidistant is affine equivariant, the sets $\Eq_{\lambda}\Big(A(\mathcal{F}_0\cup\mathcal{F}_1)\Big)$ and $A\Big(\Eq_{\lambda}(\mathcal{F}_0\cup\mathcal{F}_1)\Big)$ coincide. Let us consider the coordinate system described in Figure \ref{PictureThmParallelogramGeneralization}(ii).

Let $\ell_i$ be the length of $\mathcal{F}'_i$ for $i=0,1$. Let $[0,\ell_0]\ni s\mapsto f(s)=\left(f_1(s),f_2(s)\right)$ and $[0,\ell_1]\ni t\to g(t)=\left(g_1(t),g_2(t)\right)$ be arc length parameterizations of $\mathcal{F}'_0=A(\mathcal{F}_0)$ and $\mathcal{F}'_1=A(\mathcal{F}_1)$, respectively, such that:
\begin{enumerate}[(a)]
\item $f(0)=g(\ell_1)=(0,0)$, $g(0)=(g_1(0),1)$, where $0<g_1(0)\leqslant \rho_{\max}$ and \linebreak $f(\ell_0)=(1,f_2(\ell_0))$, where $0<f_2(\ell_0)\leqslant \rho_{\max}$,
\item $\displaystyle\frac{df}{ds}(0)=[1,0], \frac{df}{ds}(\ell_0)=[0,1], \frac{dg}{dt}(0)=[-1,0], \frac{dg}{dt}(\ell_1)=[0,-1]$.
\end{enumerate}

\begin{figure}[h]
\centering
\includegraphics[scale=0.33]{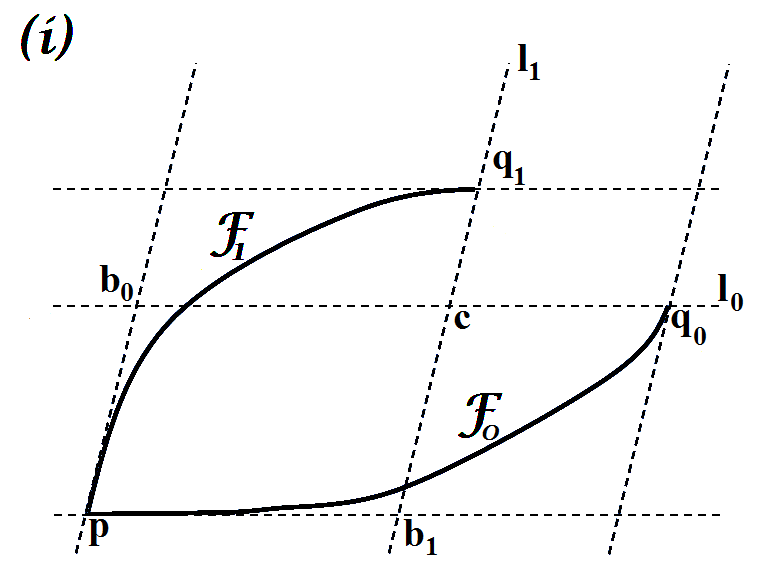}
\includegraphics[scale=0.17]{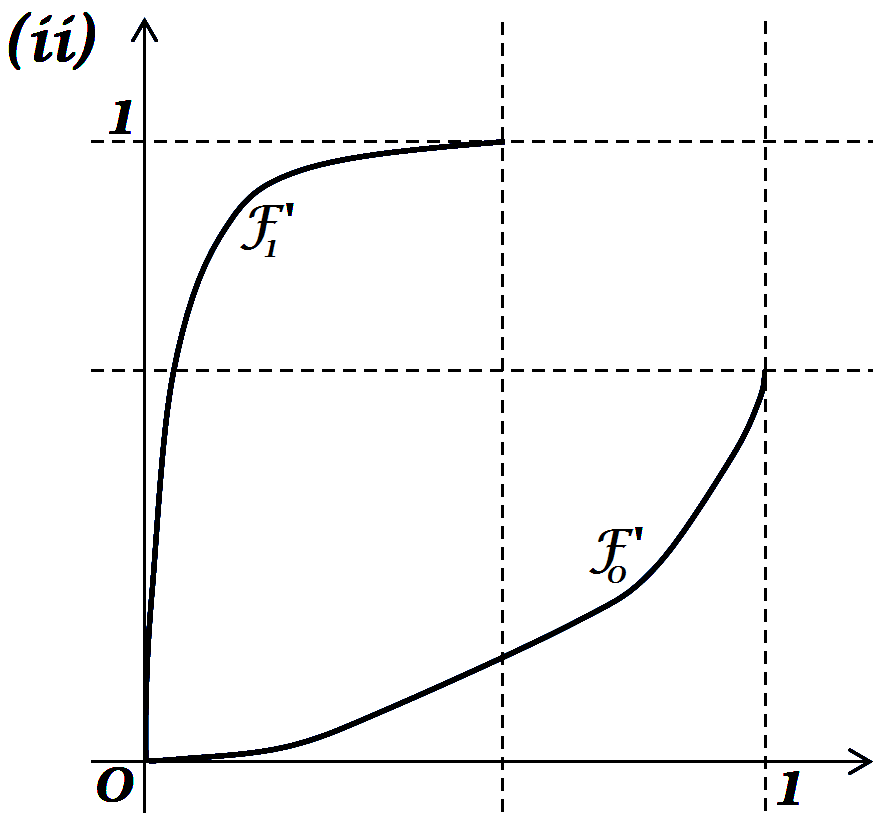}
\caption{}
\label{PictureThmParallelogramGeneralization}
\end{figure}

From the implicit function theorem there exists a smooth function $t:[0,\ell_0]\to [0,\ell_1]$ such that  
\begin{align}
\label{EqParallel}
\frac{df}{ds}(s)=-\frac{dg}{dt}(t(s)).
\end{align}
It means that $f(s),g(t(s))$ is a parallel pair.

Let $\displaystyle\varrho=\frac{\lambda}{1-\lambda}$. Then $\varrho\in[\rho_{\max}, 1]$. By Proposition \ref{PropSingularPointOfEq} the set $\Eq_{\lambda}(\mathcal{F}'_0\cup\mathcal{F}'_1)$ is singular if $\displaystyle\frac{\kappa_{\mathcal{F}'_0}(f(s))}{\kappa_{\mathcal{F}'_1}(g(t(s)))}=\varrho$ for some $s\in[0,\ell_0]$.

By (\ref{EqParallel}) we get that $\displaystyle t'(s)=\frac{\kappa_{\mathcal{F}'_0}(f(s))}{\kappa_{\mathcal{F}'_1}(g(t(s)))}$. We will show that $t'(s)=\varrho$ for some $s\in[0,\ell_0]$. Let us assume that $t'(s)\neq \varrho$ for all $s\in[0,\ell_0]$.

By (\ref{EqParallel}) we get that
\begin{align}\label{EquationGeGo}
g(\ell_1)-g(0)=\int_0^{\ell_1}\frac{dg}{dt}(t)dt=\int_0^{\ell_0}t'(s)\frac{dg}{dt}(t(s))ds=-\int_0^{\ell_0}t'(s)\frac{df}{ds}(s)ds.
\end{align}

Let us assume that $\displaystyle t'(s)>\varrho$ for all $s\in[0,\ell_0]$. 

At the first component of (\ref{EquationGeGo}) we have 
\begin{align*}
g_1(0)=\int_0^{\ell_0}t'(s)\frac{df_1}{ds}(s)ds>\int_0^{\ell_0}\varrho\cdot \frac{df_1}{ds}(s)ds=\varrho.
\end{align*}
Then $g_1(0)>\varrho$ which is impossible, since $g_1(0)\le \rho_{\max}$ and $\varrho\in[\rho_{\max}, 1]$.

If we assume that $\displaystyle t'(s)<\varrho$ for all $s\in[0,\ell_0]$, then in a similar way, we get $1<\varrho^2$, which also is impossible, since $\varrho\in[\rho_{\max}, 1]$ and $\rho_{\max}>0$.

Therefore there exists $s\in[0,\ell_0]$ such that $t'(s)=\varrho$ which ends the proof.

\end{proof}

\begin{figure}[h]
\centering
\includegraphics[scale=0.30]{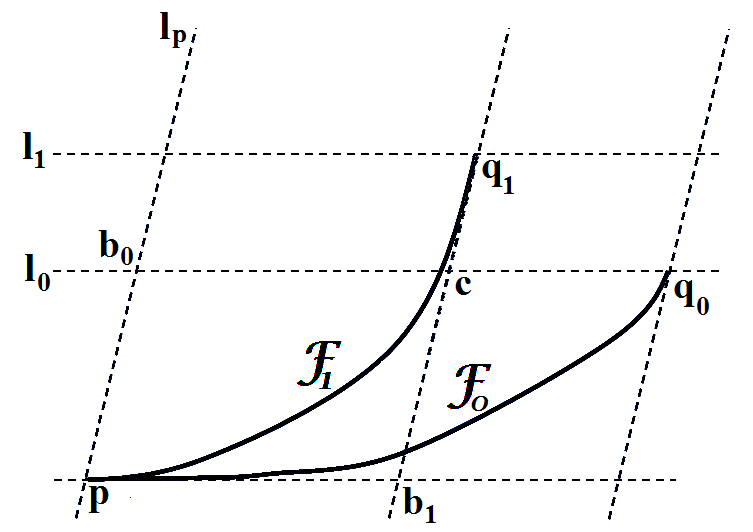}
\caption{}
\label{PictureThmParallelogramGeneralization2}
\end{figure}

In the similar way we get the two following propositions.

\begin{prop}\label{PropSingPointsGeneralization2}
Let $\mathcal{F}_0$ and $\mathcal{F}_1$ be embedded regular curves with endpoints $p, q_0$ and $p, q_1$, respectively. Let $l_0$ be the line through $q_0$ parallel to $T_{p}\mathcal{F}_0$, $l_1$ be the line through $q_1$ parallel to $T_{p}\mathcal{F}_1$ and $l_p$ be the line through $p$ parallel to $T_{q_0}\mathcal{F}_0$. Let $c=l_0\cap T_{q_1}\mathcal{F}_1$, $b_0=l_0\cap l_p$ and $b_1=T_p\mathcal{F}_1\cap T_{q_1}\mathcal{F}_1$. Let us assume that
\begin{enumerate}[(i)]
\item $T_{p}\mathcal{F}_0 = T_{p}\mathcal{F}_1$ and $T_{q_0}\mathcal{F}_0 \| T_{q_1}\mathcal{F}_1$,
\item the curvature of $\mathcal{F}_i$ for $i=0,1$ does not vanish at any point,
\item absolute values of rotation numbers of $\mathcal{F}_0$ and $\mathcal{F}_1$ are the same and smaller than $\displaystyle\frac{1}{2}$,
\item for every point $a_i$ in $\mathcal{F}_i$, there is exactly one point $a_{j}$ in $\mathcal{F}_{j}$ such that $a_i,a_{j}$ is a parallel pair for $i\neq j$,
\item $\mathcal{F}_0$, $\mathcal{F}_1$ are curved in the same side at every parallel pair $a_0, a_1$ such that $a_i\in\mathcal{F}_i$ for $i=0,1$.
\end{enumerate}

Let $\rho_{\max}$ (resp. $\rho_{\min}$) be the maximum (resp. the minimum) of the set \linebreak $\displaystyle\left\{\frac{c-b_1}{q_1-b_1}, \frac{c-b_0}{q_0-b_0}\right\}.$

If $\rho_{\max}<1$ then the set $\Eq_{\lambda}(\mathcal{F}_0\cup\mathcal{F}_1)$ has a singular point for every \linebreak $\displaystyle\lambda\in\left(-\infty, -\frac{\rho_{\max}}{1-\rho_{\max}}\right]\cup\left[\frac{1}{1-\rho_{\max}},\infty\right)$.

If $\rho_{\min}>1$ then the set $\Eq_{\lambda}(\mathcal{F}_0\cup\mathcal{F}_1)$ has a singular point for every \linebreak $\displaystyle\lambda\in\left(-\infty, \frac{1}{1-\rho_{\min}}\right]\cup\left[-\frac{\rho_{\min}}{1-\rho_{\min}},\infty\right)$.
\end{prop}

Arcs satisfying assumptions of Proposition \ref{PropSingPointsGeneralization2} are illustrated in Figure \ref{PictureThmParallelogramGeneralization2}.

\begin{prop}\label{PropSingPointsGeneralization3}
Let $\mathcal{F}_0$, $\mathcal{F}_1$ be embedded curves with endpoints $p, q_0$ and $p, q_1$, respectively. Let $l$ be the line through $p$ and $q_1$ and let $c=l\cap T_{q_0}\mathcal{F}_0$. Let us assume that
\begin{enumerate}[(i)]
\item $T_{p}\mathcal{F}_0\| T_{q_0}\mathcal{F}_0\| T_{p}\mathcal{F}_1\| T_{q_1}\mathcal{F}_1$,
\item the curvature of $\mathcal{F}_i$ does not vanish for $i=0,1$,
\item the absolute value of the rotation number of $\mathcal{F}_i$ is equal to $\displaystyle\frac{1}{2}$ for $i=0,1$. 
\end{enumerate}

Let $\displaystyle\rho=\left|\frac{q_1-p}{c-p}\right|$. If $\mathcal{F}_0$, $\mathcal{F}_1$ are curved in the different sides at every parallel pair $a_0, a_1$ such that $a_0\in\mathcal{F}_0$, $a_1\in\mathcal{F}_1$ then the set $\Eq_{\lambda}(\mathcal{F}\cup\mathcal{G})$ has a singular point for $\displaystyle\lambda\in\left\{\frac{1}{1+\rho}, \frac{\rho}{1+\rho}\right\}$.

If $\rho\neq 1$ and $\mathcal{F}_0$, $\mathcal{F}_1$ are curved in the same side at every parallel pair $a_0, a_1$ such that $a_0\in\mathcal{F}_0$, $a_1\in\mathcal{F}_1$ then the set $\Eq_{\lambda}(\mathcal{F}_0\cup\mathcal{F}_1)$ has a singular point for \linebreak $\displaystyle\lambda\in\left\{\frac{-1}{\rho-1},\frac{\rho}{\rho-1}\right\}$.
\end{prop}

\begin{figure}[h]
\centering
\includegraphics[scale=0.27]{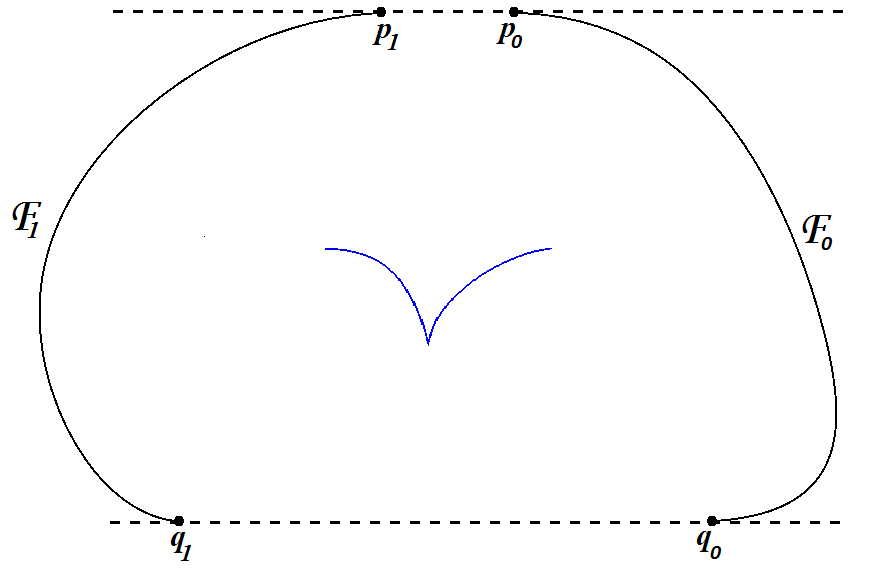}
\caption{Arcs $\mathcal{F}_0$, $\mathcal{F}_1$, the Wigner caustic of $\mathcal{F}_0\cup\mathcal{F}_1$ and the tangent lines to $\mathcal{F}_0$ and to $\mathcal{F}_1$ (the dashed ones).}
\label{PictureNonConvexLoops}
\end{figure}

\begin{rem}\label{CorGeneralizationsSingularPoints}
Let $\M$ be a smooth closed regular curve. If there exists $p$ in $\mathbb{R}^2$ and arcs $\mathcal{F},\mathcal{G}$ of $\M$ such that $\mathcal{F},\tau_p\left(\mathcal{G}\right)$ fulfill the assumptions of one of Proposition \ref{ThmSingPointsGeneralization1}, Proposition \ref{PropSingPointsGeneralization2} or Proposition \ref{PropSingPointsGeneralization3}, where $\tau_{p}$ is the translation by $p$, then the affine $\lambda$--equidistant of $\M$ has a singular point for $\lambda$ described in the above propositions. 
\end{rem}

Let us notice that curves $\mathcal{F}_0$ and $\tau_{p_0-p_1}\left(\mathcal{F}_1\right)$ in Figure \ref{PictureNonConvexLoops} satisfy Proposition \ref{PropSingPointsGeneralization3}. In that case by Remark \ref{CorGeneralizationsSingularPoints} the Wigner caustic of $\mathcal{F}_0\cup\mathcal{F}_1$ has a singular point.

Since the curve presented in Figure \ref{PictureRosettesLoops}(i) has exactly three convex and three non--convex loops (see Figure \ref{FigLoopsCC1}),   its Wigner caustic has at least $6$ singular points by Theorem \ref{CorWCLoop}. The curve presented in Figure \ref{PictureRosettesLoops}(ii) has exactly three convex loops (see \ref{FigLoopsCC2}(i)) so its Wigner caustic has at least $3$ singular points. Despite the fact that this curve has no non--convex loops we can apply Proposition \ref{PropSingPointsGeneralization3} and get that the Wigner caustic of this curve has at least $3$ extra singular points (see Figure \ref{FigLoopsCC2}(ii)). Hence the Wigner caustic of this curve has at least 6 singular points. 

\begin{figure}[h]
\centering
\includegraphics[scale=0.15]{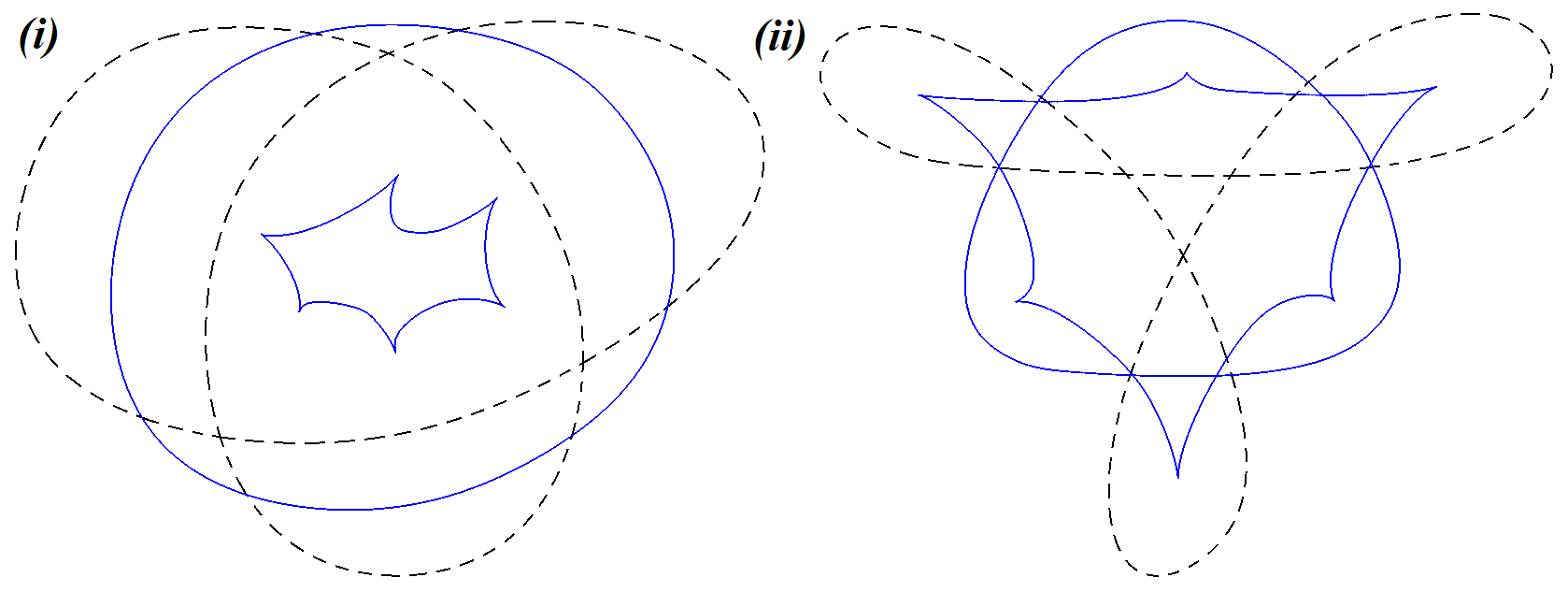}
\caption{A smooth regular curves with non--vanishing curvature (the dashed lines) and their Wigner caustics.}
\label{PictureRosettesLoops}
\end{figure}

\begin{figure}[h]
\centering
\includegraphics[scale=0.15]{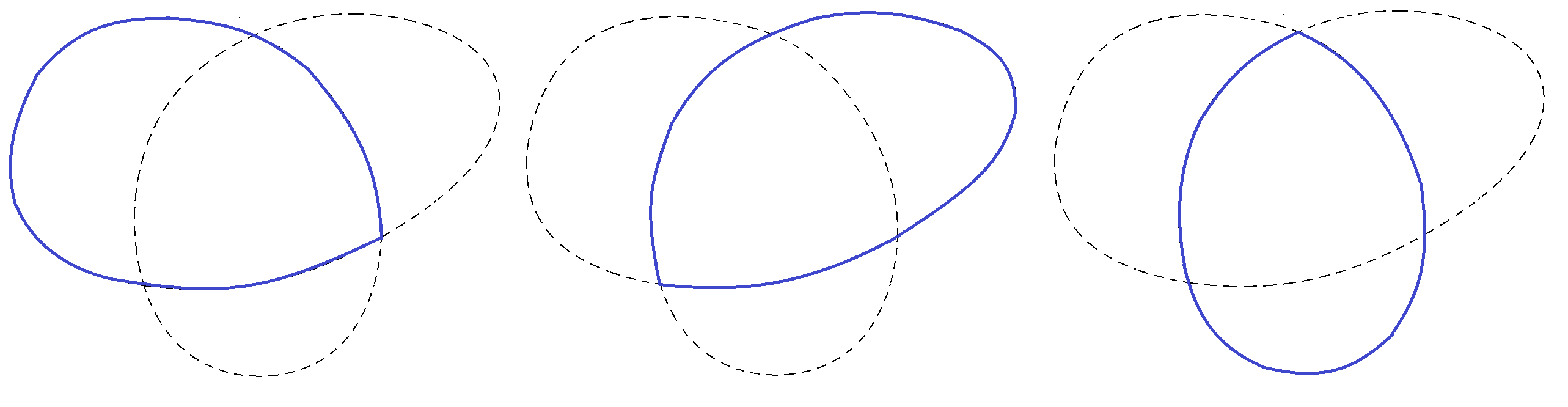}
\caption{Convex (the continuous line) and non-convex (the dased line) loops of a curve presented in Figure \ref{PictureRosettesLoops}(i).}
\label{FigLoopsCC1}
\end{figure}

\begin{figure}[h]
\centering
\includegraphics[scale=0.17]{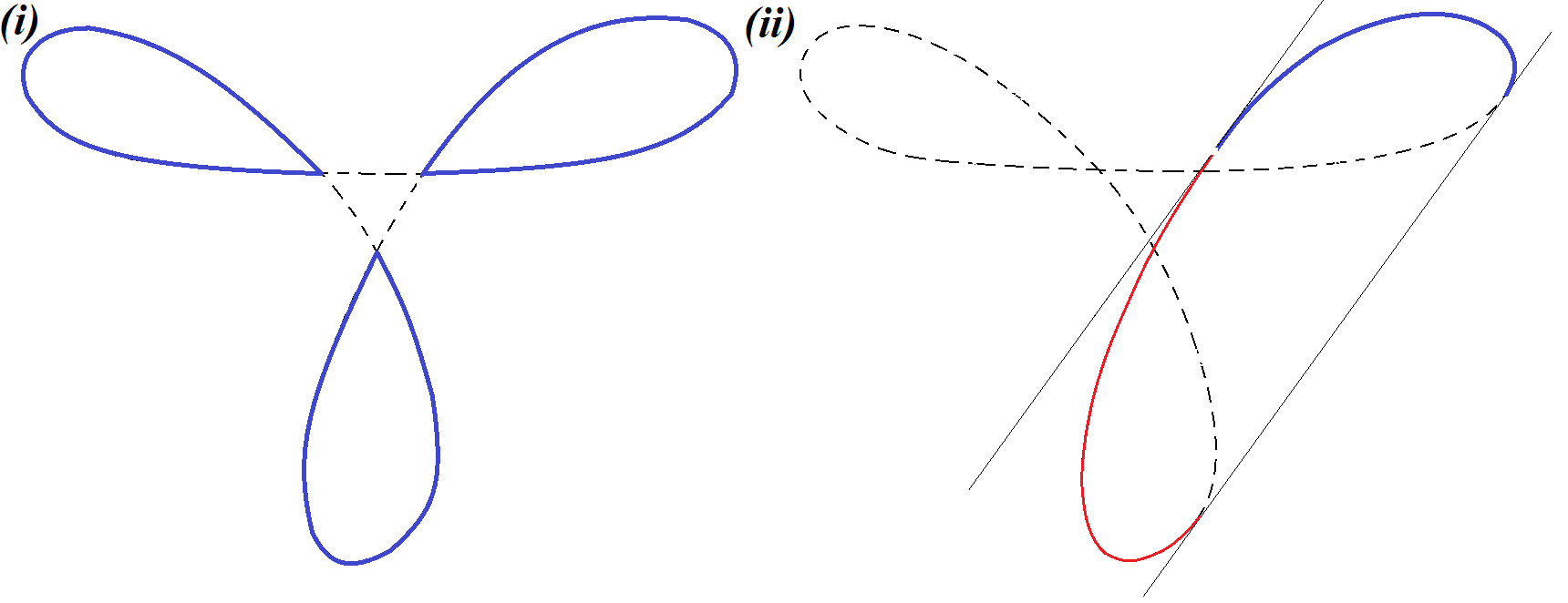}
\caption{(i) Convex loops (the continuous line) of a curve presented in Figure \ref{PictureRosettesLoops}(ii), (ii) arcs of a curve (the continuous lines) which satisfy the assumptions of Proposition \ref{PropSingPointsGeneralization3}.}
\label{FigLoopsCC2}
\end{figure}

\bibliographystyle{amsalpha}

\end{document}